\scrollmode

\documentclass[5p]{elsarticle}

\usepackage{amsmath}
\usepackage{amsthm}
\usepackage{amssymb}
\usepackage{hyperref}
\usepackage{tikz}

\newtheorem{lemma}{Lemma}[section]
\newtheorem{theorem}[lemma]{Theorem}

\newtheorem{proposition}[lemma]{Proposition}

\theoremstyle{definition}
\newtheorem{example}[lemma]{Example}
\newtheorem{definition}[lemma]{Definition}
\newtheorem{remark}[lemma]{Remark}

\numberwithin{equation}{section}

\newcommand{\leaveout}[1]{}

\begin{document}

\title{Optimal approximation of discrete-time multirate systems on Hilbert spaces}

\author{Mikael Kurula}
\ead{Mikael.Kurula@abo.fi}
\address{\AA bo Akademi Mathematics, Domkyrkotorget 1, 20500 \AA bo, Finland}

\begin{abstract}
We study discrete-time $(m,n)$-multirate systems on separable Hilbert spaces, solving the problem of approximating such a system by one which has a shorter multirate period $(m/q,n/q)$, optimally in the Hilbert-Schmidt norm. We work in the state-space setting, providing two state-space representations of the optimal approximant which are expressed in terms of a state-space representation of the original system.
\end{abstract}

\maketitle


\section{Introduction}

Control theory for periodic linear systems remains an active and interesting research field. Examples of classical applications that have motivated the development of the theory are vibration alleviation in motors and helicopter rotors, satellite attitude control, economics and finance \cite{BiCoBook}, modeling of rotating devices in electrical engineering \cite{ToHi15}, and multirate signal processing \cite{Vaid93,ChQiu97}. A recent, important application is developing control theory for the emerging smart electrical grid \cite{RYCNdSDC13}. Multirate systems are a generalization of periodic systems where, intuitively speaking, the input part of the system has a period different from the output part of the system.

In this paper, we solve the problem of approximating an $(m,n)$-multirate system in discrete time by a $(m',n')$-multirate system, where $m'<m$, but $m'/n'=m/n$. As measure of distance between two systems of this class, we use the Hilbert-Schmidt norm of the difference of the input/output maps of the systems, an extension of the classical $H^2$ norm to periodic systems, originating in \cite{BaPe92}.

This problem was solved for finite-dimensional systems in a purely input/output setting in \cite{MeCh00}, and here we extend their work to the case of $(m,n)$-multirate systems on separable Hilbert spaces, giving state-space representations of the optimal approximant. Moreover, we give a minimal, self-contained, and rigorous exposition which uses a standard periodic central system instead of switched LTI (linear, time-invariant) systems.

Pivotal steps in the development are:
\begin{enumerate}
\item formulating an $(m,n)$-multirate system $\Sigma$ in terms of a central, periodic system, an \emph{upsampler} (also called \emph{expander}) and a \emph{downsampler} (a \emph{decimator}),
\item describing the structure of the harmonic transfer function, which was pioneered by Wereley in his PhD thesis \cite{WereleyPhd} in the case of periodic systems, in the present setting with a multirate system $\Sigma$,
\item extending the Hilbert-Schmidt norm to the Hilbert-space setting, where one cannot simply take the ele\-ment-wise scalar $H^2$ norm,
\item solving the approximation problem in the frequency domain, using harmonic transfer functions, and
\item giving state-space representations of an optimal approximant $\widetilde\Sigma$, in terms of a state-space representation of the original system $\Sigma$.
\end{enumerate}

Recently, state-space representations for an optimal \emph{LTI} approximation of finite-dimensional \emph{periodic} systems, were obtained by Toivonen and Hietarinta in \cite{ToHi15}, using the Floquet transformation and techniques closely related to those used in this paper. However, not every periodic discrete-time system has a Floquet transform, see \cite{vDs93,GKK99}. We give analogues of some of the results in \cite{ToHi15} that apply in the general case. As expected, our conclusions are weaker than in \cite{ToHi15}, but on the other hand, we consider approximants of arbitrary shorter period, and for more general multirate systems. The presentation here uses ideas from \cite{ChQiu97}, with additional inspiration from \cite{MeCh00}, but here we treat the infinite-dimensional case, in a more mathematical style.

In \S\ref{sec:multirate}, we define a multirate system, and the \emph{frequency lifting technique} \cite{BiCoBook} is the topic of \S\ref{sec:transfer}; see in particular Thm \ref{thm:freqlift}. Finally, in \S\ref{sec:approx}, we solve the approximation problem, in both the frequency and the time domain.

\section{Multirate state-space systems}\label{sec:multirate}

Before we proceed to define a multirate state-space system, we need to define the $q$-\emph{upsampling operator}
$$
	(\uparrow_{q}\!v)_t:=\left\{\begin{aligned}
		v_{t/q},&\quad t/q\in\mathbb Z, \\ 
		0,&\quad t/q\not\in\mathbb Z,
	\end{aligned}\right.
$$
and the $q$-\emph{downsampling operator} $(\downarrow_{q}\!v)_t=v_{qt}$, for any sequence $v$, for any $q\in\mathbb Z^+:=\{1,2,\ldots\}$, and for all $t\in\mathbb Z$ such that these formulas make sense. 

We will study time-varying systems with the structure
\begin{equation}\label{eq:isoreporig}
	\left\{\begin{aligned}
		x_{t+1}&=A_t\,x_t+B_t\,u_{t}^\circ, \\
		y_{t}^\circ&=C_t\,x_t+D_t\,u_t,\qquad t\in\mathbb T, \\ 
		u^\circ&=\,\uparrow_p u,\quad
		y=\,\downarrow_q y^\circ,
	\end{aligned}\right.
\end{equation}
where $A_t$, $B_t$, $C_t$ and $D_t$ are sequences of bounded linear operators that map between the separable Hilbert spaces $U$ (the \emph{input space}), $X$ (the \emph{state space}), and $Y$ (the \emph{output space}), 
$$
	\begin{bmatrix} A_t&B_t\\C_t&D_t\end{bmatrix}:
	\begin{bmatrix} X\\U\end{bmatrix}\to
	\begin{bmatrix} X\\Y\end{bmatrix},\quad t\in\mathbb T,
$$ 
$p,q\in\mathbb Z^+$, and $\mathbb T\subset\mathbb Z$ is some appropriate time set. The system with input $u$ and output $y$ is referred to as the \emph{physical system}, while the system with input $u^\circ$ and output $y^\circ$ is called the \emph{central system} or \emph{underlying system}.

For any sequence $v$ indexed by $\mathbb Z$ or $\mathbb N_0:=\{0,1,\ldots\}$, we define the \emph{backward shift} operator $\sigma$ by $(\sigma v)_k:=v_{k+1}$, for all $k$ in the index set. Next we denote by $H$ the linear operator which, in \eqref{eq:isoreporig} with $\mathbb T=\mathbb N_0$, maps the input sequence $u\in U^{\mathbb N_0}$ into the output sequence $y\in Y^{\mathbb N_0}$ that corresponds to the \emph{initial state} $x(0)=0$. We will now find general sufficient conditions for the state-space system \eqref{eq:isoreporig} to satisfy the characteristic condition $\sigma^mH=H\sigma^n$, i.e., that shifting the input $n$ steps backwards in time gives the same result as instead shifting the output $m$ steps, when the initial state is in both cases equal to zero. 

Since $\uparrow_p\!\sigma^n=\sigma^{pn}\!\uparrow_p$, the signal $u^\circ$ is shifted  $pn$ steps if $u$ is shifted $n$ steps, and this should more or less have the same effect as shifting $y^\circ$ a total of $qm$ steps. This is naturally achieved if $pn=qm$ and the central system is $pn$-periodic, i.e., $A_{t+pn}=A_t$ for all $t,t+pn\in\mathbb T$, and similar for $B_t$, $C_t$ and $D_t$. The smallest rates for the upsampler and downsampler, which in general work in this setup, are obtained when $pn=qm={\rm lcm}\,(m,n)$. If we set $c:=\mathrm{gcd}\,(m,n)$, where $m,n\in\mathbb Z^+$, then $m/n=\overline m/\overline n$, with $\overline m:=m/c$ and $\overline n:=n/c$ co-prime. Then, moreover, ${\rm lcm}\,(m,n)=m\overline n$, i.e., we should choose $p=\overline m$ and $q=\overline n$. 

\begin{definition}
We call the system \eqref{eq:isoreporig}, where $p=\overline m$, $q=\overline n$ and the central system is $m\overline n$-periodic, an \emph{(m,n)-multirate (state-space) system}.
\end{definition}

It is easily verified that an $(m,n)$-multirate system is also $(km,kn)$-multirate for all $k\in\mathbb Z^+$, but all such systems have a unique smallest multirate pair $(m',n')$ with $m'/n'=m/n$ and $m'\leq m$, $n'\leq n$. Multirate systems are sometimes referred to as $(m,n)$-periodic systems. Indeed, if $n=m$, then $\overline m=\overline n=1$ and the upsampler and downsampler in \eqref{eq:isoreporig} both reduce to identity operators, so that the multirate system reduces to a standard discrete-time $m$-\emph{periodic} system. If additionally $n=m=1$, then we have a \emph{linear time-invariant (LTI)} system. The term \emph{multirate} refers to real-time applications. Indeed, let the signal $u$ be updated at a rate $f_0$ Hz. Then the $m\overline n$-periodic central system on the first two lines of \eqref{eq:isoreporig} is updated at a rate of $\overline mf_0$ Hz, and the output $y$ is then updated at the rate $\overline mf_0/\overline n$, i.e., the full system operates at two or three different rates if $n\neq m$.

\section{Frequency lifting}\label{sec:transfer}

The approximation problem will later be solved using the \emph{frequency lifting} (or \emph{alias component analysis}) technique; see \cite[\S6.4]{BiCoBook}, \cite{WereleyPhd} or \cite{ChQiu97}.

\begin{definition}\label{def:EMP}
By a \emph{$(T,z)$-exponentially modulated polynomial (EMP)}, we mean a function
\begin{equation}\label{eq:EMP}
	v_t=z^{-t}\,\sum_{k=0}^{T-1} \widehat v_k\,e^{2\pi jtk/T},
		\qquad t\in\mathbb Z,
\end{equation}
where $z\in\mathbb C\setminus\{0\}$, $j$ denotes the imaginary unit, and the coefficients $\widehat v_k$ are from some arbitrary vector space $V$. 
\end{definition}

EMPs allow us to extend the transfer function of an LTI system (the case $T=1$) to what is called the \emph{harmonic transfer function} (or the \emph{alias component representation}) of a periodic system. The first step is to observe some fundamental properties of EMPs. In this paper, we identify sequences with column vectors, $v=\mathrm{col}_{k\in\mathcal I}\,(v_k)$, where $\mathcal I\subset\mathbb Z$ is some appropriate index set.

A straightforward inversion of \eqref{eq:EMP} leads to
\begin{equation}\label{eq:EMPinv2}
	\widehat v_k=\frac 1T\sum_{t=0}^{T-1} {v_t}{z^t}\,e^{-2\pi jtk/T},\qquad k=0,\ldots,T-1,
\end{equation}
and so the EMP-Fourier coefficients $\widehat v_k$ are uniquely determined by $v$ in \eqref{eq:EMP}. Moreover, $v$ is $(T,z)$-EMP if and only if $t\mapsto z^t\,v_t$ is $T$-periodic, and $\widehat v_k$ is the standard $T$-periodic DFT of $t\mapsto z^t\,v_t$. Therefore, we denote \eqref{eq:EMP}--\eqref{eq:EMPinv2} as $\widehat v=\mathcal F_{T,z} v$, in particular $\mathcal F_{T,1}$ is the standard $T$-periodic DFT. Furthermore, a $(T,z)$-EMP $v$ is uniquely determined by its restriction to any subset containing a set $\{\ell,\ldots,\ell+T-1\}$, $\ell\in\mathbb Z$, and we identify $v$ with all such restrictions, hence also calling all these restrictions $(T,z)$-EMPs. In particular, we identify a $T$-periodic sequence with its restriction to $\{0,\ldots, T-1\}$. The output of $\mathcal F_{T,z}$ is the $(T,z)$-EMP with indices $t=0,\ldots,T-1$, i.e., by \eqref{eq:EMP}, 
$$
	\mathcal F_{T,z}^{-1}=\mathrm{diag}\,(z^{-t}I_V)_{0\leq t <T}\,\mathcal F_{T,1}^{-1}.
$$

\begin{lemma}\label{lem:EMP}
The following are true: 
\begin{enumerate}
\item Let $q$ be a positive integer. If $v$ is $(T,z)$-EMP then $\uparrow_q\!v$ is $(q T,w)$-EMP for any number $w$ such that $w^q=z$. Also, $\downarrow_q\!v$ is $(T/q,z^q)$-EMP, if $T/q\in\mathbb Z$.

\item For all positive integers $q$, 
$$
\begin{aligned}
	\mathcal F_{q T,z}\!\!\uparrow_q\,=& \frac1q\Pi_{T,1,q}\,\mathcal F_{T,z^q},\qquad\text{with}\\
	\Pi_{T,z,q}:=&\mathrm{col}\,(z^{kT}I_T)_{k=0}^{q-1},\quad
I_T:=\mathrm{diag}\,(I_V)_{k=0}^{T-1}.
\end{aligned}
$$

\item For all positive integers $q$, 
$$
	\mathcal F_{T,z^q}\!\downarrow_q\,=\Pi_{T,1,q}^*\,\mathcal F_{q T,z}.
$$

\item For all positive integers $q$,
$$
\begin{aligned}
\uparrow_q\!\mathcal F_{T,z}\,=&\,\mathcal F_{q T,z}\,\Pi_{T,1/z,q}\qquad\text{and} \\
\downarrow_q\!\mathcal F_{q T,z}\,=&\,\frac1q\,\mathcal F_{T,z}\,\Pi_{T,\overline z,q}^*.
\end{aligned}
$$
\end{enumerate}
\end{lemma}

Hence, upsampling in time (frequency) domain corresponds to extension by copying in frequency (time) domain, and downsampling in one domain is decimation by summation in the other.

\begin{proof}
By \eqref{eq:EMPinv2} it holds that $\widehat v=\mathcal F_{q T,z}\!\uparrow_q \!w$ if and only if for all $k=0,\ldots,q T-1$,
$$
\begin{aligned}	
	\widehat v_k&=\frac 1{q T}\sum_{t=0}^{q T-1} {(\uparrow_q \!w)_t}{z^t}\,e^{-2\pi jtk/q T} \\
	&=\frac 1q \frac 1T\sum_{s=0}^{T-1} {w_s}{z^{q s}}\,e^{-2\pi jsk/T} 
	= \frac 1q  (\mathcal F_{T,z^q }w)(k\,\mathrm{mod}\,T),
\end{aligned}
$$
which proves item two and the first part of item one.

From \eqref{eq:EMP}, we get for $\widehat w=\mathcal F_{q T,z}w$ and $t\in\mathbb Z$ that
$$
\begin{aligned}
	(\downarrow_q \!w)_t&=w_{q t} = z^{-q t}\,\sum_{k=0}^{q T-1} \widehat w_k\,e^{2\pi jq tk/q T} \\
	&= z^{-q t}\,\sum_{k=0}^{T-1} e^{2\pi jtk/T}\,
		\sum_{\ell=0}^{q -1} \widehat w_{k+\ell T}\\
	&= (\mathcal F_{T,z^q }^{-1} \Pi_{T,1,q }^*\widehat w)_t,
\end{aligned}
$$
and items one and three are established. The proof of item four is analogous to items two and three, with the roles of \eqref{eq:EMP} and \eqref{eq:EMPinv2} swapped.
\end{proof}

Let $A_t$ be $T$-periodic and write its Fourier expansion
\begin{equation}\label{eq:sysfourier}
	A_t = \sum_{k =0}^{T-1} \widehat A_k \,e^{2\pi j tk /T},
	\qquad t\in\mathbb Z.
\end{equation}
By \eqref{eq:EMPinv2}, the Fourier coefficients $\widehat A_k=(\mathcal F_{T,1}A)_k$ are bounded linear operators on $X$; here $A=\mathrm{col}_{t\in\mathbb Z}\, (A_t)$. Next define the \emph{Toeplitz transform with period $T$} of $A_t$,
\begin{equation}\label{eq:AtoepD}
	\mathcal T_T(A):=\left[\begin{matrix}
		\widehat A_0 & \widehat A_{T-1}  & \ldots & \widehat A_{1} \\
		\widehat A_1 & \widehat A_0 & \ldots & \widehat A_{2} \\ 
		\vdots & \vdots & \ddots & \vdots \\ 	
		\widehat A_{T-1} & \widehat A_{T-2} &\cdots & \widehat A_0
	\end{matrix}\right],
\end{equation}
which is a bounded linear operator on $X^T$. Clearly, the first column in $\mathcal A_T$ is the Fourier transform of $A_t$, and the next column is obtained by applying the forward $T$-periodic shift $\sigma_T^{-1}$, where
$$
	\sigma_T \left[\begin{matrix}v_0\\v_1\\\vdots\\v_{T-1}\end{matrix}\right]
	:=\left[\begin{matrix}v_1\\\vdots\\v_{T-1}\\v_0\end{matrix}\right];
$$
also note that
\begin{equation}\label{eq:sigmaAscKomm}
	\sigma_T\mathcal T_T(A)=\mathcal T_T(A)\sigma_T.
\end{equation}

\begin{remark}\label{rem:SwissCheese}
If $A_t$ is in fact $T/q$-periodic, where $q$ is a positive integer, then $A_t$, $0\leq t\leq T-1$, consists of $q$ copies of some $T/q$-periodic $a_t$, i.e., 
$$
	\mathrm{col}_{0\leq t<T}\,(A_t)=\Pi_{T/q,1,q}\,\mathrm{col}_{0\leq t<T/q}\,(a_t),
$$
and then item 4 of Lemma \ref{lem:EMP} gives that
$$
	\mathcal F_{T,1}A=\mathcal F_{T,1}\Pi_{T/q,1,q}a=
	\uparrow_q\mathcal F_{T/q,1} a,
$$
i.e.,
$$
	\mathrm{col}_{0\leq k<T}\,(\widehat A_k)=\uparrow_q\!\mathrm{col}_{0\leq k<T/q}\,(\widehat a_k).
$$
In particular, all Fourier coefficients $\widehat A_k$, whose indices $k$ are not multiples of $q$, are zero in this case, and from \eqref{eq:AtoepD},
\begin{equation}\label{eq:SwissInt}
\begin{aligned}
	\mathcal T_T(A)\!\uparrow_q&=\uparrow_q\!\mathcal T_{T/q}(a) \qquad\text{and}\\
	\downarrow_q\!\mathcal T_T(A)&=\mathcal T_{T/q}(a)\!\downarrow_q.
\end{aligned}
\end{equation}

Hence, $A_t$ is $T/q$-periodic if and only if the Toeplitz transform $\mathcal T_T(A)$ of $A_t$ has only zero entries except possibly in constant bands at distance $q$, with these bands centered on the main diagonal. Then the Toeplitz matrix corresponding to period $T/q$ is obtained by removing zero bands from the Toeplitz matrix of period $T$ and truncating after $T/q$ rows, so that the resulting operator is square.

The above characterization of shorter periodicity is applicable to all the operator sequences $A_t$ -- $D_t$ of a periodic system and it gives rise to an obvious characterization of when a full $T$-periodic system is in fact periodic with a shorter period $T/q$. In particular, a system is LTI if and only if all its four Toeplitz transforms $\mathcal A - \mathcal D$ are block diagonal; this follows by taking $q=T$.
\end{remark}

\begin{remark}\label{rem:lifting}
We now explain the main reason for the interest in the Toeplitz transform \eqref{eq:AtoepD}. Pick the $T$-periodic sequence $A_t$ of bounded operators on $X$ and the $X$-valued $(T,z)$-EMP $x_t$ arbitrarily. Set $\mathbf x:=\mathrm{col}\,(\mathcal F_{T,z}\,x)$, a column with $T$ entries. Combining \eqref {eq:sysfourier} with \eqref{eq:EMP} for $x$, we get
\begin{equation}\label{eq:prodhat}
	A_t\,x_t=z^{-t}\,\sum_{r=0}^{T-1}e^{2\pi j tr/T}\sum_{\ell=0}^{T-1} 
	\widehat A_{r-\ell} \,\widehat x_\ell,
\end{equation}
and this equals $x_{t+1}$ if and only if
$$
	z^{-t}\,\sum_{r=0}^{T-1}e^{2\pi j tr/T}\left(
	\sum_{\ell=0}^{T-1} \widehat A_{r-\ell} \,\widehat x_\ell
		- \frac{e^{2\pi jr/T}}{z}\widehat x_r\right)=0.
$$
Assuming this for all $0\leq t<T$, we get from \eqref{eq:EMPinv2} that
\begin{equation}\label{eq:FinvAx}
	e^{2\pi jr/T}\widehat x_r = 
	z\sum_{\ell=0}^{T-1} \widehat A_{r-\ell} \,\widehat x_\ell,
	\qquad 0\leq r<T;
\end{equation}
then in fact $x_{t+1}=A_t\,x_t$ for all $t\in\mathbb Z$ by the above.

The condition \eqref{eq:FinvAx} can be written more compactly as $	\mathcal N_T\,\mathbf x=z\,\mathcal T_T(A)\,\mathbf x$, where $\mathcal N_T$ is the unitary operator
\begin{equation}\label{eq:Ndef}
	\mathcal N_T:=\mathrm{diag}\,(\epsilon^kI_X)_{k=0}^{T-1}, \quad \epsilon:=e^{2\pi j/T}.
\end{equation}
Assuming that $T,q,T/q\in\mathbb Z^+$, one easily verifies that
\begin{equation}\label{eq:NSwissInt}
	\mathcal N_T\!\uparrow_q=\uparrow_q\!\mathcal N_{T/q}\qquad\text{and}\qquad
	\downarrow_q\!\mathcal N_T=\mathcal N_{T/q}\!\downarrow_q.
\end{equation}
\end{remark}

We next introduce the \emph{harmonic resolvent set} for an operator $\mathcal A$ on $X^T$:
$$
	\rho_{T}(\mathcal A):=\{z\in\mathbb C \mid 
		\text{$\mathcal N_T-z\mathcal A$ has a bounded inverse}\}.
$$
As $\mathcal N:=\mathcal N_T$ has the bounded inverse $\mathcal N^*$ and
\begin{equation}\label{eq:ressplit}
	\mathcal N-z\mathcal A=z\mathcal N(I_T/z-\mathcal N^*\mathcal A)
		=z\,(I_T/z-\mathcal A\mathcal N^*)\mathcal N,
\end{equation}
where $I_T$ is the identity on $X^T$, it holds that $z\in\rho_{T}(\mathcal A)\setminus\{0\}$ if and only if $1/z$ is in the common, standard resolvent set of $\mathcal N^*\mathcal A$ and $\mathcal A\,\mathcal N^*$. It is always the case that $0\in\rho_T(\mathcal A)$.

\begin{proposition}\label{prop:resinv}
The harmonic resolvent set has the following properties:
\begin{enumerate}
\item The set $\rho_T(\mathcal A)$ is open and it contains a neighborhood of the origin; in particular it is nonempty.

\item If $\mathcal A=\mathcal T_T(A)$ for some sequence $A_t$, then we have the rotational symmetry $e^{2\pi j/T}\rho_{T}(\mathcal A)=\rho_{T}(\mathcal A)$. 
\end{enumerate}
\end{proposition}

\begin{proof}
The harmonic resolvent set inherits the property of being open from the standard resolvent set by the discussion before the proposition, and clearly $0\in\rho_T(\mathcal A)$. This proves the first claim. 

By the assumption on $\mathcal A$ and \eqref{eq:sigmaAscKomm}, $\sigma_T\mathcal A=\mathcal A\sigma_T$, and by the definition \eqref{eq:Ndef} of $\mathcal N_T$, moreover
$
	\sigma_T\mathcal N_T=\epsilon\,\mathcal N_T\sigma_T.
$
These equalities imply that
\begin{equation}\label{eq:resintertw}
	\epsilon\,\mathcal N_T-z\mathcal A=
	\sigma_T\big(\mathcal N_T-z\mathcal A\big)\sigma_T^{-1},
\end{equation}
since $\sigma_T$ and its inverse are bounded, and so $z/\epsilon\in\rho_T(\mathcal A)$ if and only if $z\in\rho_T(\mathcal A)$.
\end{proof}

We are now in a good position to describe the frequency-domain behaviour of a multirate system.

\begin{theorem}\label{thm:freqlift}
Let $\Sigma$ be an $(m,n)$-multirate system as in \eqref{eq:isoreporig}, and let the input signal $u$ be $(n,z^{\overline m})$-EMP with $0\neq z\in\rho_{m\overline n}(\mathcal A)$. Define $\mathcal A:=\mathcal T_{m\overline n}(A)$, $\mathcal B:=\mathcal T_{m\overline n}(B)$, $\mathcal C:=\mathcal T_{m\overline n}(C)$ and $\mathcal D:=\mathcal T_{m\overline n}(D)$, and set $\mathbf u:=\mathcal F_{n,z^{\overline m}}u$.
\begin{enumerate}
\item There is a unique initial state $x_0$ such that the corresponding signals $u^\circ$, $x$ and $y^\circ$ of the central system of $\Sigma$, with $x(0)=x_0$, are $(m\overline n,z)$-EMP, and the downsampled external output $y$ is $(m,z^{\overline n})$-EMP. 

\item For $x_0$ as in item one, for $\mathbf u^\circ:=\mathcal F_{m\overline n,z}u^\circ$, $\mathbf x:=\mathcal F_{m\overline n,z}x$, $\mathbf y^\circ:=\mathcal F_{m\overline n,z}y^\circ$, and $\mathbf y:=\mathcal F_{m,z^{\overline n}}y$, it holds that
\begin{equation}\label{eq:freqlift}
	\left\{\begin{aligned} \mathcal N_{m\overline n}\,\mathbf x &=z\mathcal A\mathbf x+z\mathcal B\mathbf u^\circ,\\
	\mathbf y^\circ &=\mathcal C\mathbf x+\mathcal D\mathbf u^\circ,\\
		\mathbf u^\circ&=\Pi_{n,1,\overline m}\,\mathbf u/\overline m, \\
		 \mathbf y&=\Pi_{m,1,\overline n}^*\, \mathbf y^\circ.
\end{aligned}\right.
\end{equation}
Conversely, if $\mathbf u^\circ$, $\mathbf x$, $\mathbf y^\circ$ and $\mathbf y$ satisfy \eqref{eq:freqlift}, then the corres\-ponding time domain sequences satisfy \eqref{eq:isoreporig} and $x(0)$ equals the sum of the elements in $\mathbf x$.

\item For $z\in\rho_{m\overline n}(\mathcal A)$, $\mathbf u$ uniquely determines $\mathbf u^\circ$, $\mathbf x$, $\mathbf y^\circ$ and $\mathbf y$ in \eqref{eq:freqlift}. In particular, the map $\mathbf u\mapsto \mathbf y$ equals
\begin{equation}\label{eq:Ttransf}
	{\mathcal G}(z):=\Pi_{m,1,\overline n}^*\big(z\mathcal C(\mathcal N_{m\overline n} -z\mathcal A)^{-1}\mathcal B+\mathcal D\big)\Pi_{n,1,\overline m}/\overline m.
\end{equation}
\end{enumerate}
\end{theorem}
\begin{proof}
Due to Lemma \ref{lem:EMP}, $u^\circ$ is $(m\overline n,z)$-EMP, and if $y^\circ$ is of the same type, then $y$ is $(m,z^{\overline n})$-EMP. In order to prove the existence of $x_0$ in item one, set $\mathbf u^\circ:=\mathcal F_{m\overline n,z}u^\circ$ and define $\mathbf x':=z(\mathcal N_{m\overline n}-z\mathcal A\mathbf )^{-1}\mathcal B\mathbf u^\circ$, $\mathbf y':=\mathcal C \mathbf x'+\mathcal D\mathbf u^\circ$, $x':=\mathcal F_{m\overline n,z}^{-1}\,\mathbf x'$ and $y':=\mathcal F_{m\overline n,z}^{-1}\,\mathbf y'$. Then $x'$ and $y'$ are $(m\overline n,z)$-EMP and we next prove that these solve the first two equations in \eqref{eq:isoreporig}. Indeed, the first line of \eqref{eq:freqlift} holds with $\mathbf x'$ in place of $\mathbf x$, and we may use $\begin{bmatrix}A_t&B_t\end{bmatrix}$ and $\left[\begin{smallmatrix}x'_t\\u^\circ_t\end{smallmatrix}\right]$ instead of $A_t$ and $x_t$ in \eqref{eq:prodhat} with $T=m\overline n$, and continue the argument as in Rem.\ \ref{rem:lifting}, to obtain that $x'$ satisfies the first line of \eqref{eq:isoreporig}. Using \eqref{eq:prodhat} once more, for $\begin{bmatrix}C_t&D_t\end{bmatrix}$ this time, we get that $y'$ satisfies the second equation in \eqref{eq:isoreporig}. All the properties of $x_0:=x'(0)$ stated in item one, apart from uniqueness, have been shown.

Now let $u^\circ$, $x$ and $y^\circ$ be any $(m\overline n,z)$-EMPs with $z\in\rho_{m\overline n}(\mathcal A)$, and define $\mathbf u^\circ:=\mathcal F_{m\overline n,z}u^\circ$, $\mathbf x:=\mathcal F_{m\overline n,z}x$ and $\mathbf y^\circ:=\mathcal F_{m\overline n,z}y^\circ$. In fact, the argument in the previous paragraph can be slightly strengthened to prove that the first two equations of \eqref{eq:isoreporig} hold if and only if the first two equations of \eqref{eq:freqlift} hold. By Lemma \ref{lem:EMP}, equation three in \eqref{eq:isoreporig} is also equivalent to equations three and four in \eqref{eq:freqlift}. According to \eqref{eq:EMP},  the initial state $x_0=x(0)$ is uniquely determined as the sum of the elements in $\mathbf x$. Item two is established.

Now let $x_0$ be such that $x$ in \eqref{eq:isoreporig} is $(m\overline n,z)$-EMP. Then $\mathbf x:=\mathcal F_{m\overline n,z}x$ satisfies the first line of \eqref{eq:freqlift}, and since $z\in\rho_{m\overline n}(\mathcal A)$, necessarily $\mathbf x$ equals $\mathbf x'$ in paragraph one of the proof. Thus $\mathbf x$ is uniquely determined by $u^\circ$ and $z$, and the uniqueness in item one is now clear, as is also item three.
\end{proof}

We call \eqref{eq:freqlift} the \emph{frequency lifting} of \eqref{eq:isoreporig}, and ${\mathcal G}(\cdot)$ is called the \emph{(harmonic) transfer function} of $\Sigma$. Conversely, we say that \eqref{eq:isoreporig} is a \emph{realization} of $\mathcal G$. A \emph{periodic} system has $\overline m=\overline n=1$, $m=n$ and ${\mathcal G}(z):=z\mathcal C(\mathcal N_{m} -z\mathcal A)^{-1}\mathcal B+\mathcal D$. If moreover $m=n=1$, then $\mathcal G$ reduces to the standard transfer function $G(z)=zC(I_X-zA)^{-1}B+D$.

The values of the harmonic transfer function $\mathcal G$ of an $(m,n)$-multirate system are bounded linear operators from $U^n$ into $Y^m$, both separable Hilbert spaces by assumption, and it is analytic on a neighborhood of the origin; see Prop.\ \ref{prop:resinv}. Moreover, $\mathcal G$ has the following structure:

\begin{theorem}\label{thm:toep} 
Let ${\mathcal G}$ be the transfer function of an $(m,n)$-multirate system $\Sigma$, pick $z\in\rho_{m\overline n}(\mathcal A)$, where $\mathcal A:=\mathcal T_{m\overline n}(A)$, and denote the operator on row $k$ in column $0$ of ${\mathcal G}$ by $G_k$. 
\begin{enumerate}
\item Letting $\epsilon:=e^{2\pi j/m\overline n}$, we have ${\mathcal G}(z)$ equal to 
$$
\left[\begin{matrix}
G_0(z) & G_{m-1}(z/\epsilon)  & \ldots & G_{1-n}(z/\epsilon^{n-1}) \\ 
G_1(z) & G_0(z/\epsilon) & \ldots & G_{2-n}(z/\epsilon^{n-1}) \\ 
\vdots & \vdots &  & \vdots \\ 	
G_{m-1}(z) & G_{m-2}(z/\epsilon) & \cdots & G_{n}(z/\epsilon^{n-1})
\end{matrix}\right],
$$
where indices of $G$ in the last column are modulo $m$.

\item Assume now that $\Sigma$ is also $(\widetilde m,\widetilde n)$-multirate, where $\widetilde m=\widetilde c\,\overline m$ and $\widetilde n=\widetilde c \,\overline n$ with $q:=c/\widetilde c\in\mathbb Z^+$, and denote the $(\widetilde m,\widetilde n)$-period transfer function of $\Sigma$ by $\widetilde{\mathcal G}$. Then the domain of $ \mathcal G$ is contained in the domain of $\widetilde{\mathcal G}$ and for all $z\in\rho_{m\overline n}(\mathcal A)$, we have $\mathcal G(z)\!\uparrow_q=\uparrow_q\!\widetilde{\mathcal G}(z)$, and in particular $\widetilde{\mathcal G}(z)=\downarrow_q\!\mathcal G(z)\!\uparrow_q$.
\end{enumerate}
\end{theorem}

Given the first column of  ${\mathcal G}$, we hence get the next column by applying $\sigma_m^{-1}$ and dividing the argument by $\epsilon$. Thus ${\mathcal G}$ is uniquely determined by its first column. In general, we say that $\mathcal E$ is the $q$-\emph{inflate} of $\widetilde{\mathcal E}$ if $\mathcal E\!\uparrow_q=\uparrow_q\!\widetilde{\mathcal E}$.

\begin{proof}
Fix $z\in\rho_{m\overline n}(\mathcal A)$ arbitrarily and denote the element on row $k$ in column $\ell$ of ${\mathcal G}(z)$ by $G_{k,\ell}(z)$. Column 0 of $\mathcal G(z)$ is correct by definition, and we next prove that $G_{k+1,\ell+1}(z)=G_{k,\ell}(z/\epsilon)$, with indices modulo $m$ and $n$, respectively. Denote by ${\mathbf I}^{(\ell,n)}$ the column vector in $U^n$ whose all elements are $0_U$, except for the entry $I_U$ in position $\ell$, so that $\mathbf I^{(0,n)}=\Pi_{1,0,n}$ with $V=U$. By definition, $G_{k+1,\ell+1}(z)$ is on row $k+1$ in $\mathcal G(z)\,\mathbf I^{(\ell+1,n)}$. 

First we prove the second of the two intertwinements 
\begin{equation}\label{eq:ShiftPi}
\begin{aligned}
	\Pi_{n,1,\overline m}\,\sigma_n&=\sigma_{m\overline n}\,\Pi_{n,1,\overline m}
	\qquad\text{and}\\
	\sigma_m\,\Pi_{m,1,\overline n}^*&=\Pi_{m,1,\overline n}^*\,\sigma_{m\overline n}.
\end{aligned}
\end{equation}
Indeed, for all $k=0,\ldots,m-1$ and $m\overline n$-columns $\mathbf v$,
$$
	(\Pi_{m,1,\overline n}^*\,\sigma_{m\overline n}\, v)_k =
	\sum_{\ell=0}^{\overline n-1} v_{k+1+\ell m} =
	(\sigma_m\,\Pi_{m,1,\overline n}^* v)_k.
$$
Replacing both occurences of $\sigma$ by $\sigma^*$ and $+1$ by $-1$ in the previous calculation, we get the same equality with periodic left shifts replaced by periodic right shifts; then the first line of \eqref{eq:ShiftPi} follows from the unitarity of $\sigma_T$.

Next, \eqref{eq:resintertw} and item 2 of Prop.\ \ref{prop:resinv} imply that
\begin{equation}\label{eq:resshiftint}
	(\mathcal N_{m\overline n}-z\mathcal A)^{-1}\sigma_{m\overline n}
	= \epsilon\,\sigma_{m\overline n}
		(\mathcal N_{m\overline n}-\epsilon\, z\mathcal A)^{-1}.
\end{equation}
Setting $\mathcal G^\circ(z):=z\mathcal C(\mathcal N_{m\overline n} -z\mathcal A)^{-1}\mathcal B+\mathcal D$, where $\mathcal B:=\mathcal F_{m\overline n}(B)$, $\mathcal C:=\mathcal F_{m\overline n}(C)$ and $\mathcal D:=\mathcal F_{m\overline n}(D)$, we get from \eqref{eq:sigmaAscKomm} and \eqref{eq:resshiftint} that $\mathcal G^\circ(z)\,\sigma_{m\overline n}=\sigma_{m\overline n}\,\mathcal G^\circ(\epsilon\, z)$. Combining this with \eqref{eq:ShiftPi} and $\sigma_n\,{\mathbf I}^{(\ell+1,n)}={\mathbf I}^{(\ell,n)}$, we have
$$
\begin{aligned}
	G_{k,\ell}(z) &= \big({\mathcal G}(z)\,{\mathbf I}^{(\ell,n)}\big)_k \\
	&=\left(\Pi_{m,1,\overline n}^*\,\mathcal G^\circ (z)\,\Pi_{n,1,\overline m}\,\sigma_n{\mathbf I}^{(\ell+1,n)}\right)_k/\overline m\\
	&= \big(\sigma_m\,{\mathcal G}(\epsilon\,z)\,{\mathbf I}^{(\ell+1,n)}\big)_k
	= G_{k+1,\ell+1}(\epsilon\,z).
\end{aligned}
$$
This proves item one.

To prove claim 2, assume that $\Sigma$ is also $(\widetilde m,\widetilde n)$-multirate as in the statement. By the definition of a multirate system, the $m\overline n$-periodic central system is then in fact $\widetilde m\overline n$-periodic, and by Remark \ref{rem:SwissCheese}, it then holds that $\widehat A:=\mathcal F_{m\overline n}(A)=\uparrow_q\!\widehat a$ with $\widehat a:=\downarrow_q\!\widehat A$. Denoting $a:=\mathcal F_{\widetilde m\overline n}^{-1}\widehat a$ and $\widetilde{\mathcal A}:=\mathcal T_{\widetilde m\overline n}(a)$, we have from \eqref{eq:AtoepD} that (indices mod $m\overline n$)
$$
	(\mathcal A\!\uparrow_q\!\mathbf x)_k = 
	\sum_{\ell=0}^{m\overline n-1} \widehat A_{k-\ell}\,(\uparrow_q\!\mathbf x)_{\ell} =
	\sum_{\ell=0}^{\widetilde m\overline n-1} \widehat A_{k-\ell q}\mathbf x_{\ell} =
	(\uparrow_q\!\widetilde{\mathcal A}\mathbf x)_k,
$$
and a similar calculation gives $\downarrow_q\!\mathcal A=\widetilde{\mathcal A}\!\downarrow_q$. The corresponding intertwinements hold for $\mathcal B$, $\mathcal C$, $\mathcal D$ defined like $\mathcal A$, and combining $\mathcal A\!\uparrow_q=\uparrow_q\!\widetilde{\mathcal A}$ with \eqref{eq:NSwissInt}, we get
\begin{equation}\label{eq:resint1}
	(\mathcal N_{m\overline n}-z\mathcal A)\!\uparrow_q = 
	\uparrow_q\!(\mathcal N_{\widetilde m\overline n}-z\widetilde{\mathcal A}).
\end{equation}
This implies, for $z\in\rho_{m\overline n}(\mathcal A)$, that
$$
	I_{\widetilde m\overline n}=\downarrow_q\uparrow_q = 
	\downarrow_q\!(\mathcal N_{m\overline n}-z\mathcal A)^{-1}\uparrow_q\!(\mathcal N_{\widetilde m\overline n}-z\widetilde{\mathcal A}),
$$
which shows that $\downarrow_q\!\!(\mathcal N_{m\overline n}-z\mathcal A)^{-1}\!\uparrow_q$ is a bounded left inverse for $\mathcal N_{\widetilde m\overline n}-z\widetilde{\mathcal A}$. Starting instead with the down\-sampling version of \eqref{eq:resint1}, one gets that this left inverse is in fact a bounded inverse, proving that $\rho_{m\overline n}(\mathcal A)\subset \rho_{\widetilde m\overline n}(\widetilde{\mathcal A})$. For all $z\in\rho_{m\overline n}(\mathcal A)$, \eqref{eq:resint1} implies that $\uparrow_q$ intertwines the resolvents. From $\uparrow_q \!I_{\widetilde m}=I_m\!\uparrow_q$ and the definition of $\Pi$ in Lem.\ \ref{lem:EMP}, it follows that $\uparrow_q\!\Pi_{\widetilde m,1,\overline n}^*=\Pi_{m,1,\overline n}^*\!\uparrow_q$ and $\uparrow_q\!\Pi_{\widetilde n,1,\overline m}=\Pi_{n,1,\overline m}\!\uparrow_q$. Hence, $\mathcal G(z)\uparrow_q=\uparrow_q\widetilde{\mathcal G}(z)$, where
$$
	\widetilde{\mathcal G}(z):=\Pi_{\widetilde m,1,\overline n}^*\big(z\widetilde{\mathcal C}(\mathcal N_{\widetilde m\overline n}-z\widetilde{\mathcal A})^{-1}\widetilde{\mathcal B}+\widetilde{\mathcal D}\big)\Pi_{\widetilde n,1,\overline m}/\overline m,
$$
and applying $\downarrow_q$ from the left, we get the final equality.
\end{proof}

The particular initial time of a system is a rather arbitrary choice, and it is of interest to consider how the transfer function changes if one happens to choose a different initial time. 
It turns out that this only results in a certain type of complex rotation of the components of the transfer function, which we next describe. 

For an $(m,n)$-multirate system $\Sigma$, \eqref{eq:sysfourier} implies that
$$
	(\sigma^\tau A)_t=A_{t+\tau} = \sum_{k=0}^{m\overline n-1} (\widehat A_k\,e^{2\pi j \tau k/m\overline n})\,e^{2\pi j tk/m\overline n}.
$$
Denote $\mathcal A_{(\tau)}:=\mathcal T_{m\overline n}\,(\sigma^\tau A)$, to get from \eqref{eq:AtoepD} that
$$
	\mathcal A_{(\tau)}=\mathcal N_{m\overline n}^\tau\,\mathcal A\,\mathcal N_{m\overline n}^{-\tau},
$$
and analogously for $\mathcal B_{(\tau)}$, $\mathcal C_{(\tau)}$, and $\mathcal D_{(\tau)}$. Let $\mathcal G$ be the transfer function of $\Sigma$, and let $\mathcal G_{(\tau)}$ be the transfer function of $\Sigma$ with system operators shifted $\tau$ steps to the left. With $\mathcal G^\circ$ the transfer function of the central system, we thus get
$$
	\mathcal G^\circ_{(\tau)}(z)=\mathcal N_{m\overline n}^\tau\,\mathcal G^\circ(z)\,\mathcal N_{m\overline n}^{-\tau},\quad
	z\in\rho_{m\overline n}(\mathcal A)=\rho_{m\overline n}(\mathcal A_{(\tau)}).
$$
If we want both the up- and the downsampler to be active at the new initial time $0$, then $\tau=\overline m\,\overline n$ is the shortest possible shift in the central system. From
$$
	\mathcal N_{\overline mn}^{-\overline m\,\overline n}\,\Pi_{n,1,\overline m} = 
	\Pi_{n,1,\overline m}\,\mathcal N_n^{-\overline n},	
$$
we get 
$$
\begin{aligned}
	\mathcal G_{(\overline m\,\overline n)}(z) &= \Pi_{m,1,\overline n}^*\,\mathcal G^\circ_{(\overline m\,\overline n)}(z)\,\Pi_{n,1,\overline m}/\overline m \\\
	&= \Pi_{m,1,\overline n}^*\,\mathcal N_{m\overline n}^{\overline m\,\overline n}\,\mathcal G^\circ(z)\,\mathcal N_{\overline mn}^{-\overline m\,\overline n}\,\Pi_{n,1,\overline m}/\overline m \\
	&= \mathcal N_m^{\overline m}\,\mathcal G(z)\,
	\mathcal N_n^{-\overline n}.
\end{aligned}
$$

We now proceed to optimal approximation.

\section{Optimal approximation with shorter period}\label{sec:approx}

The $H^2$ norm is commonly used for finite-dimensional LTI systems, and this norm was extended to periodic systems in \cite{BaPe92}. The extended norm, which is of Hilbert-Schmidt type, essentially consists in taking the element-wise $H^2$ norm of the first column of the harmonic transfer function, which is matrix valued in \cite{BaPe92}. Thus the formulation requires some modification in order to extend to the separable Hilbert-space case; see also \cite[\S\S3.4--5]{BalaBook76} for more background on the Hilbert-Schmidt norm.

Let $\mathcal U$ be some separable Hilbert space and let the sequence $e_\ell$, $\ell\in\Lambda_1$, be an orthonormal basis for $\mathcal U$. With $\mathbf I$ introduced in the proof of Thm \ref{thm:toep}, the sequence $e_{\ell,k}:=\mathbf I^{(k,n)}e_\ell$ is then an orthonormal basis for $\mathcal U^n$, with index set $\Lambda_2:=\Lambda_1\times\{0,\ldots,n-1\}$, such that the Hilbert-Schmidt norm of $\mathcal G:\mathcal U^n\to H^2(\mathbb D;Y)^m$ satisfies (item 1 of Thm \ref{thm:toep})
\begin{equation}\label{eq:HS}
\begin{aligned}
	|{\mathcal G}|^2:=& \,\frac 1n\|\mathcal G\|_{HS}^2=
	\frac 1n\sum_{(\ell,k)\in\Lambda_2} \|{\mathcal G}e_{\ell,k}\|^2_{H^2(\mathbb D;Y)^m} \\
	=& 
	\frac 1n\sum_{(l,k)\in\Lambda_2,\,i=0,\ldots,m-1} \| G_i(\cdot/\epsilon^k)\,e_{\ell}\|^2_{H^2(\mathbb D;Y)}  \\ 
	=&	\,\sum_{i=0}^{m-1} \|G_i\|^2_{HS},
\end{aligned}		
\end{equation}
where $\|G_i\|_{HS}$ is the Hilbert-Schmidt norm of $G_i:\mathcal U\to H^2(\mathbb D;Y)$, assuming that all of these $G_i$ have finite Hilbert-Schmidt norms. The factor $1/n$ scales the Hilbert-Schmidt norm of $\mathcal G$, so that it becomes invariant under inflation, as can be seen in the last line of \eqref{eq:HS}. 

The rest of the paper is dedicated to solving the following approximation problem in frequency and time domain: For a given $(m,n)$-multirate system with transfer function $\mathcal G$ and a pair $(\widehat m,\widehat n)$ with $\widehat m/\widehat n=m/n$, $0<\widehat m<m$, find an $(\widehat m,\widehat n)$-multirate system with transfer function $\widetilde{\mathcal G}$, such that and the Hilbert-Schmidt norm $|\mathcal G-\widetilde{\mathcal G}|$ is minimized. In Thm \ref{thm:NewOpt}, we first find the optimal $\widetilde{\mathcal G}$, and in Thm \ref{thm:Approx}, we show that $\widetilde {\mathcal G}$ is indeed the harmonic transfer function of an $(\widehat m,\widehat n)$-multirate system.

\begin{theorem}\label{thm:NewOpt}
Let $\Sigma$ be $(m,n)$-multirate and set $c:=\mathrm{gcd}\,\,(m,n)$, $m=c\overline m$ and $n=c\overline n$. Any $|\cdot|$-optimal $(\widehat c\,\overline m,\widehat c\,\overline n)$-multirate approximant $\widetilde\Sigma$ of $\Sigma$, $1\leq \widehat c<c$ has transfer function $\downarrow_{c/\widetilde c}\mathcal G\!\uparrow_{c/\widetilde c}$ (modulo inflation), where $\widetilde c:=\mathrm{gcd}\,\,(c,\widehat c)$.
\end{theorem}

Please observe that the minimizing transfer function of minimal periodicity is unique up to inflation, but this implies no uniqueness of state-space representation, unless further minimality assumptions are made. Later, we will produce two $(\widetilde c\,\overline m,\widetilde c\,\overline n)$-multirate realizations of $\downarrow_{c/\widetilde c}\mathcal G\!\uparrow_{c/\widetilde c}$, thus proving that an optimal $(\widehat c\,\overline m,\widehat c\,\overline n)$-multirate approximant can in fact be found among the $(\widetilde c\,\overline m,\widetilde c\,\overline n)$-multirate systems; note that often $\widetilde c<\widehat c$.

\begin{proof}
Let $\widehat\Sigma$ be an arbitrary $(\widehat c\,\overline m,\widehat c\,\overline n)$-multirate system. Following the proof of \cite[Theorem 1]{MeCh00}, we start off by observing that the error system $\Sigma':=\Sigma-\widehat\Sigma$ is $(C\overline m,C\overline n)$-multirate, where $C:=\mathrm{lcm}\,\,(c,\widehat c)=c\,\widehat c/\widetilde c$. By item 2 of Thm \ref{thm:toep}, the first column of the error transfer function $\mathcal G'$ is
\begin{equation}\label{eq:optimizethis}
	\mathcal G'\,\mathbf I^{(0,C\overline n)} =
	\uparrow_{\widehat c/\widetilde c} \mathcal G\,\mathbf I^{(0,c\overline n)}\,
		-\uparrow_{c/\widetilde c} \widehat{\mathcal G}\,\mathbf I^{(0,\widehat c\,\overline n)},
\end{equation}
and by \eqref{eq:HS}, we minimize $|\mathcal G'|$ by choosing $\widehat{\mathcal G}$ so as to make as many of the $G'_i$ as possible zero, $0\leq i\leq C\overline m-1$. 

For $i=\ell c/\widetilde c$, $\ell\in\mathbb Z$, which is not an integer multiple of $\widehat c/\widetilde c$, we have $G'_i=-\widehat G_\ell$ which should be chosen as zero. Then $\widehat G_\ell\neq0$ only if $\ell c/\widetilde c=k\,\widehat c/\widetilde c$ for some $k\in\mathbb Z$, i.e., for 
$$
	\ell=\frac{\widehat c/\widetilde c}{c/\widetilde c}\,k\in
	\frac {\widehat c}{\widetilde c}\,\mathbb Z,
$$
since $\widehat c/\widetilde c$ and $c/\widetilde c$ are co-prime by the definition of $\widetilde c$, and necessarily $\ell\in\mathbb Z$. Then
$$
	\widehat{\mathcal G}\,\mathbf I^{(0,\widehat c\,\overline n)}=
	\uparrow_{\widehat c/\widetilde c}\,
	\widetilde{\mathcal G}\,\mathbf I^{(0,\widetilde c\,\overline n)}
$$
with $\widetilde {\mathcal G}:=\,\downarrow_{\widehat c/\widetilde c} \widehat{\mathcal G}\uparrow_{\widehat c/\widetilde c}$, and because upsamplers commute and $\mathbf I^{(0,c\overline n)}=\uparrow_{c/\widetilde c}\mathbf I^{(0,\widetilde c\,\overline n)}$, \eqref{eq:optimizethis} can be written
$$
	\mathcal G'\,\mathbf I^{(0,C\overline n)} =
	\uparrow_{\widehat c/\widetilde c} \big(
	\mathcal G\uparrow_{c/\widetilde c}
		-\uparrow_{c/\widetilde c}
		\widetilde{\mathcal G}\big)\,\mathbf I^{(0,\widetilde c\,\overline n)}.
$$
This column has the maximal number of zero elements if we choose $\widetilde {\mathcal G}:=\,\downarrow_{c/\widetilde c}\mathcal G\uparrow_{c/\widetilde c}$. This is the optimal $\widetilde{\mathcal G}$, and the optimal $\widehat{\mathcal G}$ is the $\widehat c/\widetilde c$-inflate of $\widetilde{\mathcal G}$.
\end{proof}

\begin{remark}\label{rem:intsample} 
By \eqref{eq:Ttransf} and Thm \ref{thm:NewOpt}, and the way that $\uparrow_q$ and $\downarrow_q$ interact with $\Pi$, see the end of the proof of Theorem \ref{thm:toep}, $|\cdot|$-optimal approximations are obtained simply by downsampling the $m\overline n$-periodic central system of $\Sigma$:
$$
\begin{aligned}
	\widetilde {\mathcal G} &= \,
	\downarrow_q\!\mathcal G\!\uparrow_q \\
	&= \Pi_{\widetilde m,1,\overline n}^*\downarrow_q\!\big(z\mathcal C(\mathcal N_{m\overline n} -z\mathcal A)^{-1}\mathcal B+\mathcal D\big)\!\uparrow_q\!\Pi_{\widetilde n,1,\overline m}/\overline m.
\end{aligned}
$$
\end{remark}

Next we give an example which shows that, in general, an optimal approximation of $\Sigma$ with shorter period must have a larger state space than $\Sigma$, even for $\Sigma$ \emph{periodic} with $A_t$ \emph{constant}. However, unlike in the case of \emph{time} lifting, we will below obtain a state-space representation of the optimal approximant, which has the same input and output spaces as the original system $\Sigma$.

\begin{example}\label{ex:firstex}
Consider the following $2$-periodic system $\Sigma$: Let $U:=X:=Y:=\mathbb C$, $A_t:=1/2$, $D_t:=0$, for $t=0,1$,
$$
	B_t:=\left\{\begin{aligned} 2,&\quad t=0, \\ 6,&\quad t=1,\end{aligned}\right.\qquad\text{and}\qquad
	C_t:=\left\{\begin{aligned} -1,&\quad t=0, \\ 3,&\quad t=1.\end{aligned}\right.
$$
Taking Toeplitz transforms and using \eqref{eq:Ttransf}, we get
$$
	\mathcal G(z) = \frac{4z}{\left( 2-z\right)   \left( 2+z\right) }
	\left[\begin{matrix} 4z & 6-5z \\
	  -6-5z&4z \end{matrix}\right],
$$
and by Theorem \ref{thm:NewOpt}, (the transfer function of) the $|\cdot|$-optimal LTI approximation of $\Sigma$ 
thus has two distinct poles. No LTI system with state space $X=\mathbb C$ has this property. 
\end{example}

However, it \emph{is} possible to get state-space formulas for  the optimal LTI approximant of the system in Example \ref{ex:firstex}, where the main operator $A$ of the approximant acts on the same space as the Toeplitz matrix $\mathcal A$ of the original system, $\mathbb C^2$. This follows from the final result of the paper, which provides a time-domain analogue of Theorem \ref{thm:NewOpt}. 

In order to formulate the final result, we need to introduce some notation. Let $\Sigma=(A_t,B_t,C_t,D_t)$ be an $(m,n)$ multirate system, set $\mathcal A:=\mathcal T_{m\overline n}\,(A)$, $\mathcal B:=\mathcal T_{m\overline n}\,(B)$, $\mathcal C:=\mathcal T_{m\overline n}\,(C)$, and let $q$ be a positive integer which divides both $m$ and $n$. Next denote $\epsilon:=e^{2\pi j/m\overline n}$ and set 
$$
	M:=\mathrm{diag}\,(\epsilon^kI_X)_{k=0}^{q-1}\qquad\text{and}\qquad
	\mathfrak M:=\mathrm{diag}\,\big(M)_{k=0}^{m\overline n/q-1},
$$
so that, with $\mathfrak N_{m\overline n/q}$ defined in analogy to \eqref{eq:Ndef} but with $I_X$ replaced by $I_{X^q}$, and disregarding the block subdivision,
\begin{equation}\label{eq:MscrNscr}
	\mathcal N_{m\overline n}=\mathfrak N_{m\overline n/q}\,\mathfrak M=\mathfrak M\,\mathfrak N_{m\overline n/q}.
\end{equation}

Further combine the operator entries of $\mathcal A$ into $q\times q$ blocks and denote the result by $\mathfrak A$, so that $\mathfrak A$ consists of $m\overline n/q\times m\overline n/q$ operators on $X^q$. Perform the same operation on $\mathcal B$ and $\mathcal C$ to obtain $\mathfrak B:(U^q)^{m\overline n/q}\to (X^q)^{m\overline n/q}$ and $\mathfrak C:(X^q)^{m\overline n/q}\to (Y^q)^{m\overline n/q}$. Denote by $\Uparrow_q$ the operator $\uparrow_q:U^{m\overline n/q}\to U^{m\overline n}$ reinterpreted as an operator from $U^{m\overline n/q}$ to $(U^q)^{m\overline n/q}$, and let $\Downarrow_q:(Y^q)^{m\overline n/q}\to Y^{m\overline n/q}$ be the analogous reinterpretation of $\downarrow_q:Y^{m\overline n}\to Y^{m\overline n/q}$. Finally, $\mathbf I^{m\overline n|q}_V$ denotes the $m\overline n/q$ column with $I_{V^q}$ in position zero and the zero operator on $V^q$ in the other positions.

\begin{theorem}\label{thm:Approx} 
Let $\Sigma$ be $(m,n)$-multirate with $m/q,n/q\in\mathbb Z$ for $q\in\mathbb Z^+$. The $(m/q,n/q)$-multirate system $\widetilde\Sigma$ with
\begin{equation}\label{eq:App1D2}
\begin{aligned}
	\widetilde B&:=\mathcal F_{m\overline n/q,1}^{-1}\,\mathfrak B\Uparrow_q\!\mathbf I^{m\overline n/q|1}_U, \\
	\widetilde A&:=\mathcal F_{m\overline n/q,1}^{-1}\,\mathfrak A\,\mathfrak M^{-1}\,\mathbf I^{m\overline n|q}_X, \\
	\widetilde C&:=\mathcal F_{m\overline n/q,1}^{-1}\Downarrow_q\mathfrak C\,\mathfrak M^{-1}\,\mathbf I^{m\overline n|q}_X,
	\end{aligned}
\end{equation}
and, for all $0\leq t\leq m\overline n/q-1$,
\begin{equation}\label{eq:Dapprox}
	\widetilde D_t:=(\Pi_{m\overline n/q,1,q}^*\,D)_t/q=
		\frac 1q\sum_{k=0}^{q-1} D_{t+km\overline n/q},
\end{equation}
is an $|\cdot|$-optimal $(m/q,n/q)$-multirate approximation of $\Sigma$.

The statement remains true if \eqref{eq:App1D2} is replaced by
\begin{equation}\label{eq:App1D}
\begin{aligned}
	\widetilde B&:=\mathcal F_{m\overline n/q,1}^{-1}\,\mathfrak M^{-1}\,\mathfrak B\Uparrow_q\!\mathbf I_U^{m\overline n/q|1}, \\
	\widetilde A&:=\mathcal F_{m\overline n/q,1}^{-1}\,\mathfrak M^{-1}\,\mathfrak A\,\mathbf I_X^{m\overline n|q}, \qquad\text{and}\\
	\widetilde C&:=\mathcal F_{m\overline n/q,1}^{-1}\Downarrow_q\mathfrak C\,\mathbf I_X^{m\overline n|q}.
	\end{aligned}
\end{equation}
\end{theorem}

Before we proceed to the proof, we give an example to illustrate the theorem.

\begin{example}
In this example, we will calculate a state-space realization of the $(1,2)$-multirate optimal approximant of the following $(2,4)$-multirate system $\Sigma$ with state space $X:=\mathbb C^2$ and $U:=Y:=\mathbb C$, for $t=0,1,2,3$:
$$
	A_t:=j^tI_2, \quad
	B_t:=\begin{bmatrix}1\\0\end{bmatrix},\quad C_t:=B_t^*
	\quad\text{and}\quad D_t:=\delta_{0,t},
$$
where $I_n$ is the $n\times n$ identity matrix and $\delta_{0,t}$ is the Kronecker delta. The associated Toeplitz transforms are
$$
	\mathcal A=\sigma_8^{-2},\quad
	\mathcal B=I_8\uparrow_2,\quad
	\mathcal C=\mathcal B^*
	\quad\text{and}\quad
	\mathcal D=\frac14\left[\begin{smallmatrix}
		1&1&1&1\\1&1&1&1\\1&1&1&1\\
		1&1&1&1\end{smallmatrix}\right].
$$
Using that $\mathcal N_4=\mathrm{diag}\,(1,1,j,j,-1,-1,-j,-j)$, $m=q=\overline n=n/q=2$, $\overline m=m/q=1$, and that $\Pi_{2,1,2}^*=\begin{bmatrix}I_2&I_2\end{bmatrix}$ and $\Pi_{4,1,1}=I_4$ for $U=Y=\mathbb C$, we get from \eqref{eq:Ttransf} that the harmonic transfer function of $\Sigma$ satisfies
\begin{equation}\label{eq:ExTransfDownsamp}
	\downarrow_2\mathcal G(z)\!\uparrow_2=
	\begin{bmatrix} \frac{z^4+2jz^3+2z+1}{2\, ( z^4+1) }&
	\frac{z^4-2jz^3-2z+1}{2\, ( z^4+1) }
	\end{bmatrix}.
\end{equation}

Next we compute the system \eqref{eq:App1D2}--\eqref{eq:Dapprox} and its harmonic transfer function. In this case, $M=\mathrm{diag}\,(I_2,jI_2)$, $\mathfrak M=\mathrm{diag}\,(M,M)$ and $\mathfrak N_2=\mathrm{diag}\,(I_4,-I_4)$; indeed \eqref{eq:MscrNscr} holds. Moreover, $\mathbf I^{m\overline n/q|1}_U=\left[\begin{smallmatrix}1\\0\end{smallmatrix}\right]$ and $\mathbf I^{m\overline n|q}_X=\left[\begin{smallmatrix}I_4\\0\end{smallmatrix}\right]$. We get
$$
	\widetilde A_t=\begin{bmatrix} 
		0&-j^tI_2\\I_2&0	\end{bmatrix},\quad
	\widetilde B_t=\left[\begin{smallmatrix}
		1\\0\\0\\0\end{smallmatrix}\right]=
	\widetilde C_t^*
	\quad\text{and}\quad \widetilde D_t=\frac12 \delta_{0,t},
$$
$t=0,1$, for the system $\widetilde\Sigma$. As expected, we get $\widetilde{\mathcal B}=\mathcal B\!\uparrow_2$, $\widetilde{\mathcal C}=\downarrow_2\!\mathcal C$ and $\widetilde{\mathcal D}=\downarrow_2\!\mathcal D\!\uparrow_2$, and also
$$
	\widetilde{\mathcal A} = \sigma_8^{-2}\,\mathfrak M^{-1}
	=\left[\begin{smallmatrix} 0&0&0&-jI_2\\I_2&0&0&0\\
		0&-jI_2&0&0\\0&0&I_2&0\end{smallmatrix}\right].
$$
As $\Pi_{1,1,2}^*=\begin{bmatrix}1&1\end{bmatrix}$ and $\Pi_{2,1,1}=I_2$ for $\mathbb C$, we use \eqref{eq:Ttransf} to see that the harmonic transfer function of $\widetilde\Sigma$ equals \eqref{eq:ExTransfDownsamp}.
\end{example}

\begin{proof}[Proof of Theorem \ref{thm:Approx}.]
By Rem.\ \ref{rem:intsample}, it suffices to show that the harmonic transfer function $\widetilde{\mathcal G}^\circ$ of the central system in $\widetilde\Sigma$, corresponding to the periodicity $m\overline n/q$, coincides with $\downarrow_q\!\mathcal G^\circ\!\uparrow_q$, where $\mathcal G^\circ$ is the harmonic transfer function of the central system of $\Sigma$, with period $m\overline n$. It is clear that \eqref{eq:App1D2}--\eqref{eq:Dapprox} is $m\overline n/q$-periodic and we next show that this system also realizes $\downarrow_q\!\mathcal G^\circ\!\uparrow_q$. By \eqref{eq:Dapprox} and item 4 of Lemma \ref{lem:EMP}, we have $\mathcal F_{m\overline n/q,1}\widetilde D =\,\downarrow_q\!\mathcal F_{m\overline n,1} D$, and then \eqref{eq:AtoepD} gives $\widetilde{\mathcal D}:=\mathcal T_{m\overline n/q}(\widetilde D)= \,\downarrow_q\!\mathcal D\!\uparrow_q$, where $\mathcal D:=\mathcal T_{m\overline n}(D)$.

We next establish that $\widetilde{\mathcal C}:=\mathcal T_{m\overline n/q}\,(\widetilde C)=\,\Downarrow_q\! \mathfrak C\,\mathfrak M^{-1}$. Since the right-hand side inherits the Toeplitz structure \eqref{eq:AtoepD} from $\mathcal C$, it suffices to show that the first columns of these operator matrices coincide. Indeed, by \eqref{eq:App1D2}:
$$
	\mathcal T_{m\overline n/q}\,(\widetilde C)\,
	\mathbf I_X^{m\overline n|q} = 
	\mathcal F_{m\overline n/q,1}\,\widetilde C =\,
	\Downarrow_q\!\mathfrak C\,\mathfrak M^{-1}\mathbf I_X^{m\overline n|q}.
$$
The same argument gives $\widetilde{\mathcal A}:=\mathcal T_{m\overline n/q}\,(\widetilde A)=\mathfrak A\,\mathfrak M^{-1}$ and $\widetilde{\mathcal B}:=\mathcal T_{m\overline n/q}\,(\widetilde B)=\mathfrak B\!\Uparrow_q$. Disregarding the block subdivision, we get $\Downarrow_q\mathfrak C=\downarrow_q\mathcal C$, $\mathfrak B\Uparrow_q=\mathcal B\uparrow_q$, and from \eqref{eq:MscrNscr},
$$
	\mathfrak M^{-1}(\mathfrak N_{m\overline n/q}-z\,\mathfrak A\,\mathfrak M^{-1})^{-1} = 
	(\mathcal N_{m\overline n}-z\mathcal A)^{-1}.
$$
By combining these with $\widetilde{\mathcal D}=\downarrow_q\!\mathcal D\!\uparrow_q$, one gets 
\begin{equation}\label{eq:GcircCompress}
	\widetilde{\mathcal G}^\circ=
	z\widetilde{\mathcal C}(\mathfrak N_{m\overline n/q}-z
	\widetilde{\mathcal A})^{-1}\widetilde{\mathcal B}+
	\widetilde{\mathcal D}=
	\downarrow_q\!\mathcal G^\circ\!\uparrow_q.
\end{equation}
The result has now been proved for \eqref{eq:App1D2}, and next we obtain \eqref{eq:App1D} by a similarity argument.

Arguing as in the previous paragraph and denoting the blocked Toeplitz transforms of the operators in the left-hand sides of \eqref{eq:App1D} by $\widetilde{\mathcal B}'$, $\widetilde{\mathcal A}'$ and $\widetilde{\mathcal C}'$, we get
\begin{equation}\label{eq:freqsim}
	\widetilde{\mathcal B}'=\mathfrak M^{-1}\,\widetilde{\mathcal B},\quad
	\widetilde{\mathcal A}'=\mathfrak M^{-1}\,\widetilde{\mathcal A}\,\mathfrak M
	\quad\text{and}\quad
	\widetilde{\mathcal C}'=\widetilde{\mathcal C}\,\mathfrak M.
\end{equation}
By \eqref{eq:MscrNscr}, moreover, $\mathfrak M^{-1}\,\mathfrak N_{m\overline n/q}\,\mathfrak M=\mathfrak N_{m\overline n/q}$, so that
$$
	\widetilde{\mathcal C}'(\mathfrak N_{m\overline n/q}-z\widetilde{\mathcal A}')^{-1}\widetilde{\mathcal B}'\,=\,
	\widetilde{\mathcal C}(\mathfrak N_{m\overline n/q}-z\widetilde{\mathcal A})^{-1}\widetilde{\mathcal B},
$$
i.e., the central system given by \eqref{eq:Dapprox}--\eqref{eq:App1D} has the same harmonic transfer function \eqref{eq:GcircCompress}; thus \eqref{eq:Dapprox}--\eqref{eq:App1D} is also an $(m/q,n/q)$-multirate realization of the optimal approximant $\downarrow_q\!\mathcal G^\circ\!\uparrow_q$.
\end{proof}

In general, due to the definitions of $\mathfrak M$ and the Toeplitz transform, the similarity conditions \eqref{eq:freqsim} on the Toeplitz transforms corresponds to the following similarity conditions in the time domain, for $t=0,\ldots,T-1$:
$$
	B'_t=M^{-1}B_t,\quad A'_t=M^{-1}A_tM
	\quad\text{and}\quad C'_t=C_tM.
$$

Combining Theorem \ref{thm:Approx} with Theorem \ref{thm:NewOpt}, we can now conclude that a $(\widehat c\,\overline m,\widehat c\,\overline n)$-multirate optimal approximation of a $(c\overline m,c\overline n)$-multirate system can in fact be taken to be $(\widetilde c\,\overline m,\widetilde c\,\overline n)$-multirate, where $0<\widehat c<c$ and $\widetilde c=\mathrm{gcd}\,\,(c,\widehat c)$, because we know how to construct a state-space realization of $\downarrow_{c/\widetilde c}\!\mathcal G\!\uparrow_{c/\widetilde c}$. Recall that we defined an $(m,n)$-multirate system in terms of its state-space representation in \S\ref{sec:multirate}.

\section{Final remarks and acknowledgments}

A frequency-response interpretation of the harmonic transfer function is discussed extensively in \cite{ChQiu97}, and there is an extension to possibly non-periodic time-varying systems in \cite{BGK95}.

In \cite{FFK02}, the $H^2$ norm is extended to operator-valued functions in a way differing from this paper, by using an $H^2$ space of functions whose values are bounded linear operators between separable Hilbert spaces. There the authors also introduce the so-called $\mathcal P_m$ norm, which in a sense interpolates between the $H^2$ and $H^\infty$ norms, containing the latter two as the special cases $m=1$ and $m\to\infty$, respectively.

An algebraically analogous investigation could be carried out for the continuous-time case, see \cite{WereleyPhd, BiCoBook}, but in the continuous-time case, the Fourier analysis becomes more involving \cite{ZH02,ZhouPhd} even in the finite-dimensional case, because the Toeplitz matrices then become bi-infinite and convergence issues enter the discussion. Moreover, most of the interesting applications of infinite-dimensional continuous-time systems theory are PDE models, and this implies unboundedness of at least some of the operators in the state-space representation. The theory then changes character, possibly preventing a meaningful unification of the discrete-time and continuous-time cases; see, e.g., \cite{TanabeBook, dPIchi88, ChenWeiss15}.

The author thanks Prof.\ Hannu T.\ Toivonen for inspiring discussions on the topic and pointers to the literature. The research project was in part carried out using funding from the Foundation of Ruth and Nils-Erik Stenb\"ack.

\newcommand{\etalchar}[1]{$^{#1}$}
\def\cprime{$'$} \def\cprime{$'$}
\providecommand{\bysame}{\leavevmode\hbox to3em{\hrulefill}\thinspace}
\providecommand{\MR}{\relax\ifhmode\unskip\space\fi MR }
\providecommand{\MRhref}[2]{%
  \href{http://www.ams.org/mathscinet-getitem?mr=#1}{#2}
}
\providecommand{\href}[2]{#2}

\end{document}